\documentclass[11pt]{article}
\usepackage{amsmath,amssymb,amsthm,amsfonts,epsfig,enumerate}
\usepackage{mathtools}
\usepackage{multicol}
\usepackage{pst-plot,pst-node}
\usepackage{mathrsfs,amssymb}
\usepackage{graphicx}
\usepackage{color}
\usepackage{enumerate}
\newtheorem{definition}{Definition}[section]
\newtheorem{theorem}[definition]{Theorem}

\newtheorem{remark}[definition]{Remark}
\newtheorem{proposition}[definition]{Proposition}
\newtheorem{lemma}[definition]{Lemma}

\numberwithin{equation}{section}
\DeclarePairedDelimiter\abs{\lvert}{\rvert}%
\DeclarePairedDelimiter\norm{\lVert}{\rVert}%

\makeatletter
\let\oldabs\abs
\def\abs{\@ifstar{\oldabs}{\oldabs*}}
\let\oldnorm\norm
\def\norm{\@ifstar{\oldnorm}{\oldnorm*}}
\newcommand{\al} {\alpha}

\newcommand{\De} {\Delta}

\newcommand{\ga} {\gamma}

\newcommand{\Om} {\Omega}
\newcommand{\la} {\lambda}

\newcommand{\Gr} {\nabla}
\newcommand{\no} {\nonumber}

\newcommand{\var} {\varepsilon}
\newcommand{\ra} {\rightarrow}

\newcommand{\wide}[1]{\widetilde{#1}}

\newcommand\restr[2]{{
  \left.\kern-\nulldelimiterspace 
  #1 
  \right|_{#2} 
  }}

\def\w{{\widetilde w}}

\def\dx{{\rm d}x}
\def\dy{{\rm d}y}

\def\inpr#1{\left\langle #1\right\rangle}

\def\sb2{{{\mathcal D}^{1,2}_0(B_1^c)}}

\def\wpo{{{ W}^{1,p}_{0}}(\Om)}
\def\wp{{{ W}^{1,p}_{0}}}
\def\w2{{{ W}^{1,2}_{0}}}

\def\A{{\mathcal A}}
\def\C{{\mathcal C}}

\def\E{{\mathcal E}}
\def\K{{\mathcal K}}
\def\M{{\mathcal M}}
\def\S{{\mathcal S}}

\def\R{{\mathbb R}}
\def\N{{\mathbb N}}
\def\F{{\mathcal F}}
\def\O{{\mathcal O}}
\def\({{\Big(}}
\def\){{\Big)}}

\def\ws2{{\F_{\frac{N}{2}}}}
\def\l2{{ L^{1,\;\infty}(\log L)^2}}
\def\M2{\mathcal M_{\log L}}

\title{On the structure of the second eigenfunctions of the $p$-Laplacian on a ball}
\author{T. V. Anoop \footnote{This work is funded by the
project EXLIZ - CZ.1.07/2.3.00/30.0013, which is co-financed by the European Social Fund and the state budget of the
Czech Republic.}\,,
P. Dr\'{a}bek  \footnote{The author was supported
by the
Grant Agency of Czech Republic,
Project No. 13-00863S.
\clearpage \vspace{-5mm} \hspace{+0.2mm}
$^\dag$Corresponding author, Email: pdrabek@kma.zcu.cz, Tel.: +420 377632648.
\clearpage \vspace{-5mm}\hspace{+0.2mm}
$^*$Email:\;anoop@ntis.zcu.cz,  
sasi@ntis.zcu.cz }\;,  Sarath 
Sasi \footnotemark[1]\, \\
 \small{Department of Mathematics and NTIS,
Faculty of Applied Sciences,} \\
\small{University of West Bohemia, Univerzitn\'{i} 8, 306 14 Plze\v{n}, Czech
Republic.}
}

\date{}
\begin{document}
\maketitle
\begin{abstract}
In this paper, we prove that the second eigenfunctions of the
$p$-Laplacian, $p>1$, are not radial on the unit ball in $\R^N,$ for any $N\ge 2.$  Our proof
relies on the  variational characterization of the second eigenvalue and a variant of the deformation lemma. We also construct an infinite sequence
of eigenpairs $\{\tau_n,\Psi_n\}$ such that $\Psi_n$ is nonradial and has exactly $2n$ nodal domains.
A few related open problems are also stated.
\end{abstract}
\section{Introduction}
Let $B_1\subset \R^N$ be the open unit ball centred at the
origin. We consider the following eigenvalue problem:
\begin{align}\label{evp}
 -\De_p u &= \la |u|^{p-2}u\quad \mbox{ in } B_1, \no\\
  u &= 0 \quad \mbox{ on } \partial B_1,
\end{align}
where $\De_p u:= {\rm{div}}(|\Gr u|^{p-2}\Gr u)$ is the $p$-Laplace operator with
$p>1$ and $\la$ is the spectral parameter. 
A real number $\la$ for which
\eqref{evp} admits a non-zero weak solution in $\wp(B_1)$ is called an eigenvalue of \eqref{evp} and corresponding solutions are called the eigenfunctions
associated with $\la.$ 

For $p=2,$ it is well known that the set of all eigenvalues of \eqref{evp}
can be arranged in a sequence 
   $$0<\la_1<\la_2\le \la_3\ldots \ra \infty $$
and the corresponding normalized eigenfunctions form an orthonormal basis
for the Sobolev space $\w2(B_1)$.  Further, using the Courant-Weinstein
variational principle (Theorem 7.8.14 of \cite{Drabek-book2}), these
eigenvalues can be expressed as follows:
\begin{align*}
  \la_k:= \inf_{\{u \perp \{u_1,\ldots,u_{k-1}\}, \norm{u}_2=1\}}\int_{B_1} |\Gr u|^2\, {\dx}, \;
k=1,2,3,\ldots,
\end{align*}
where $u_i$ is an eigenfunction corresponding to $\la_i.$ 
For $p\ne 2,$ using Ljusternik-Schnirelman theorem, an infinite 
sequence $\{\mu_n\}$ of eigenvalues of \eqref{evp} is   
provided in \cite{Azorero}. Possibly a different sequence $\{\la_n\}$
of variational eigenvalues of \eqref{evp} is 
provided  in \cite{Drabek-Resonance}. We stress that  a 
complete description of the set of all eigenvalues of \eqref{evp} for $p\ne
2$ is a
challenging open problem. Nevertheless, a complete description of the set
of all
radial eigenvalues $\{\ga_n\}$ (eigenvalue with a radial eigenfunction) of
\eqref{evp} is
given in \cite{DelPino}. The authors of \cite{DelPino} showed that $\la$ is a radial eigenvalue
of \eqref{evp} if and only if the following ODE has a non-zero solution:
\begin{align}\label{rad_evp}
 -\big(r^{N-1}\,|u'(r)|^{p-2}u'(r)\big)'&=\la r^{N-1}\,|u(r)|^{p-2}u(r)
\quad \mbox { in } (0,1), \no \\
 u'(0)=0 &,\quad  u(1)=0. 
\end{align}
Regardless of the methods by which the eigenvalues are obtained, one can
uniquely identify the first two eigenvalues of
\eqref{evp} as below: 
\begin{align*}
 \la_1 &=\min \{\la: \la \mbox{ is an eigenvalue of } \eqref{evp} \},\\
 \la_2 &=\min \{\la>\la_1: \la \mbox{ is an eigenvalue of } \eqref{evp} \}.
\end{align*}

It is well known that the eigenfunctions corresponding to $\la_1$ are
radial and  keep the same sign  on $B_1.$ All other eigenfunctions change its
sign on $B_1.$ The structure of the
second eigenfunctions are not well understood, except for $p=2.$ In this case,
the Fourier method for the Laplacian in the  polar coordinates gives
the precise form of the second eigenfunctions. In particular, it is evident
that the second eigenfunctions are not radial. One anticipates the same results
also for $p\neq 2.$ 

In \cite{Parini}, Parini proved that the  second eigenfunctions
are not radial in a
special case, where $B_1$ is the disc
($B_1\subset \R^2$) and $p$ is close to 1.
In \cite{Benedikt}, this
result is extended for every $p\in(1,\infty)$ using a computer aided proof.
Indeed, these methods are not readily extendable to
dimensions greater than 2. Here, we give a simple analytic proof for their
result which works in all dimensions ($N\ge 2$)
and for every $p\in(1,\infty)$. Our proof relies on the variational characterization of $\la_2$ given in
\cite{Drabek-Resonance} and a variation of the deformation lemma given in  \cite{Ghoussoub}. We also use a result from \cite{Anisa} 
that states that for a fixed $r\in (0,1),$
$$\la_1(B_1\setminus \overline{B_r(x)})\le \la_1(B_1\setminus \overline{B_r(0)}),$$
where $B_r(x) \subset B_1$ is the ball with centre $x$ and radius $r$. Now we state our main result:
\begin{theorem}\label{Theorem1}
 Let $B_1$ be the unit ball centred at the origin in $\R^N$ with $N\ge 2$ and
let
$1<p<\infty.$ Let $\la_2$ be the second eigenvalue of \eqref{evp}. Then the
eigenfunctions corresponding to $\la_2$ are not radial.
\end{theorem}

In this paper we also construct a sequence $\{\tau_n,\Psi_n\}$ of
eigenpairs of \eqref{evp} such that the eigenfunction $\Psi_n$ is nonradial and
has exactly $2n$ nodal domains. Furthermore, the sequence $\{\tau_n\}$ is
strictly increasing and unbounded. 
In fact the nodal domains can be specified using 
the spherical coordinate system for $\R^N$ which consists of a radial coordinate $r$ and
angular coordinates $\theta_1,\ldots,\theta_{N-1}$ where
$\theta_1,\ldots,\theta_{N-2} \in [0,\pi]$ and $\theta_{N-1}\in[0,2\pi).$ 
By a sector of the ball $B_1$ we mean the set $S$ given by
$S=\{x\in B_1: 0<\theta_*<\theta_{N-1}<\theta^*<2\pi \}.$
We prove the following theorem.
\begin{theorem}\label{Theorem2}
Let $B_1\subset \R^N.$ Then for each $n\in \N$  there exists an
eigenpair $\{\tau_n,\Psi_n\}$  of \eqref{evp} such that $\Psi_n$ has exactly
$2n$ nodal domains where each nodal domain is a sector with measure
$\frac{|B_1|}{2n}.$
\end{theorem}

The rest of this paper is organized as follows. In Section 2, we consider Dirichlet eigenvalue for $p$-Laplacian on a 
general domain and discuss the existence and the regularity properties of the eigenfunctions. We also discuss the variational 
characterizations of eigenvalues and state a version of the deformation lemma.  In Section 3, we give a proof for  Theorem
\ref{Theorem1}. The last section consists of a proof of Theorem \ref{Theorem2}
and some important open problems related to eigenvalues of $p$-Laplacian.

\section{Preliminary}
In this section we consider the eigenvalue problem on a  bounded domain $\Om$ in $\R^N:$ 
\begin{align}\label{genevp}
 -\De_p u &= \la |u|^{p-2}u\quad \mbox{ in } \Om, \no\\
  u &= 0 \quad \mbox{ on } \partial \Om.
\end{align}
We discuss the existence and regularity properties of the eigenfunctions of \eqref{genevp}.
If $\la$ is an eigenvalue of \eqref{genevp} and $u\in \wpo$ is an associated eigenfunction, then we have
\begin{align}\label{weak-form}
  \int_{\Om} |\Gr u|^{p-2}\Gr u \cdot \Gr v\, \dx = \la \int_{\Om} |u|^{p-2} u v \, \dx,\quad \forall\, v\in \wpo.
\end{align}
Now we consider the following two functionals on $\wpo:$
$$J(u)=\int_\Om |\Gr u|^p \dx ,\qquad G(u)=\int_\Om | u|^p \dx .$$
Using the Lagrange multiplier theorem, it can be easily verified that the critical values and critical points of $J$ on the manifold $\S=G^{-1}(1)$ satisfy \eqref{weak-form}. Indeed, the eigenvalues of \eqref{genevp} and the critical values of $J$ on $S$ are one and the same. The least critical value of $J$ on $\S$ is given by 
$$\la_1=\inf_{u\in \S} J(u).$$  In the next proposition, we list some of the important properties of $\la_1$ and the corresponding eigenfunctions. 
\begin{proposition}\label{properties}
 Let $\la_1$ be the first eigenvalue of \eqref{genevp}. Then 
  \begin{itemize}
  \item [\rm{(i)}] $\la_1$ is simple
  \item [\rm{(ii)}] any eigenfunction corresponding to $\la_1$ keeps the same sign on $\Om,$
  \item [\rm{(iii)}] any eigenfunction corresponding  to an eigenvalue $\la>\la_1$ changes its sign on $\Om,$   
  \item [\rm{(iv)}] if $\;\Om=B_r(0)$, then the eigenfunctions corresponding to $\la_1$ are radial. 
 \end{itemize}
 
\end{proposition}
\begin{proof}
 For a proof of (i) and (ii) see \cite{Lindqvist} ,  (iii) follows from Theorem 1.1 of \cite{Kawohl-Lindqvist}. Finally (iv) is evident from (i) and (iii) by noting the existence of a radial positive eigenfunction for \eqref{genevp} when $\Om=B_r(0)$.  
 \end{proof}

An infinite set of critical values of $J$ on $S$ are obtained in \cite{Azorero} using the variational methods. Their approach relies on the notion of 
Krasnoselskii genus of a symmetric closed set. For a symmetric closed subset $\A\subset \S$,  Krasnoselskii genus of $\A$ is defined as 
\begin{align*}
 \ga(\A)&:=\inf \left\{n\in \N : \exists \mbox { a continuous odd map from } \A \mbox{ into } \R^n\setminus \{0\}\right\}
\end{align*} 
with the convention $\inf\{\emptyset\}=\infty.$ For each $n\in \N,$ let 
 \begin{align*}
   \E_n&:=\left\{\A\subset \S: \A=\overline{\A},\; \A=-\A \mbox{ and } \ga(\A)\ge n\right\},\\
   \mu_n&:=\inf_{\A\in \E_n} \sup_{u\in \A} J(u).
 \end{align*}
 Then $\mu_n$ is a critical value of $J$ on $\S$ (see Proposition 5.4 of \cite{Azorero}). Possibly another set of critical values are obtained in 
 \cite{Drabek-Resonance} by considering a special collection of sets with genus $n$ in $\S.$ Note that, 
 the unit sphere $\S^{n-1}$ in $\R^n$ has genus $n$ and hence its image under an odd continuous map has the same genus. For each $n \in \N,$ let
 \begin{align*}
  \F_n&:=\left\{\A\subset \S: \;\A =h(\S^{n-1}),\; h \mbox{ is an odd continuous map from }  \S^{n-1}\ra \S \right\},\\
  \mu^*_n&:=\inf_{\A\in \F_n} \sup_{u\in \A} J(u).
 \end{align*}
Then $\mu_n^*$ is a critical value of $J$ on $\S$ (see Theorem 5 of \cite{Drabek-Resonance}). Since $\F_n \subset \E_n,$ we always have 
$\mu_n\le \mu_n^*$. It is known that $\la_i=\mu_i=\mu^*_i$ for $i=1,2.$ This result for $i=1$ follows as  the set $\{u,-u\}$ lies in 
both $\E_1$ and $\F_1$ for  $u\in \S.$ Let $u$ be an eigenfunction corresponding to $\la_2.$ Then by (ii) of Proposition \ref{properties} 
both $u^+$ and $u^-$ are nonzero. Thus the set $\A:=\left\{ a u^+ + b u^-: |a|^p\norm{u^+}_p^p+|b|^p\norm{u^-}_p^p=1 \right\}$ lies in both 
$\E_2$ and $\F_2.$ Now as $J(a u^+ + b u^-)=\la_2,$ we get $\mu_2\le \la_2$ and $\mu_2^*\le \la_2.$ Since there is no eigenvalue between 
$\la_1$ and $\la_2$, it follows that $\la_2=\mu_2=\mu_2^*.$ In particular, we have the following variational characterization of $\la_2$ that 
we use later:  
 \begin{equation} 
  \la_2=\inf_{\A\in \F_2} \sup_{u\in \A} J(u).\label{var2}
 \end{equation}
   The next proposition  is a consequence of the deformation lemma (see Lemma 3.7 of  \cite{Ghoussoub}, see also Theorem 2.1 and Remark 2.3 of \cite{Drabek-Nodal}). Note that $J\in\C^1(\wpo;\R)$ and $\S$ is a $\C^1$ manifold. Further, $J(u)=J(-u)$ and $\S=-\S.$  
 \begin{proposition}\label{deformation}
  Let $\S,J$ be as before. Let $\K$ be a compact subset of $S.$ If $\norm{J'(u)}_* \ge \var>0$ for all $u\in \K,$ then there exists a continuous one parameter family of homeomorphisms $\Psi:S\times [0,1]\ra S$ such that 
  \begin{itemize}
    \item [\rm{(i)}] $J(\Psi(u,t))\le J(u)-\var t,$ for every $u\in \K,\; t\in [0,1],$
   \item [\rm{(ii)}] $\Psi(-u,t)=-\Psi(u,t),$ for all $u\in \S,\; t\in [0,1].$
  \end{itemize}
 In particular, if $\K\in \F_n$ and $J$ has no critical point  on $\K,$ then the set $\wide{\K}= \left \{\Psi(u,1):u\in \K\right\}$ is in $\F_n$ and 
 \begin{equation}
  \sup_{u\in \wide{\K}}J(u)< \sup_{u\in{\K}}J(u). \label{downflow}
 \end{equation}
 \end{proposition}
 
 We also need the following result on the regularity of the eigenfunctions of \eqref{genevp} which is a consequence of Theorem 1 of \cite{Lieberman}.
 \begin{proposition}\label{regularity}
  Let $\Om$ be a bounded domain in $\R^N$ with smooth boundary. Let $\phi$ be an eigenfunction of \eqref{genevp}. Then there exists $\al\in(0,1)$ such that 
 $\phi\in \C^{1,\al}(\overline{\Om}).$
 \end{proposition}
\section{Radial asymmetry of the second eigenfunctions}
In this section we prove our main result. First we state a lemma that follows from Proposition 4.1 of \cite{DelPino}.
\begin{lemma}\label{radial_2}
Let $\ga_2$  be the second radial eigenvalue of \eqref{rad_evp}. Then any
radial eigenfunction corresponding to
$\ga_2$ has exactly two nodal domains - a ball and an annulus with centre 
at the origin. In particular, there exist $r\in (\frac{1}{2},1)$ such that $\la_1(B_{r}(0))=\ga_2=\la_1(B_1\setminus \overline{B_{r}(0)}).$
\end{lemma}

Now using the '$r$' given by the above lemma, we construct a special collection of sets in $\F_2.$
Let $r$ be as in Lemma \ref{radial_2}. Then for each $n\in \N\cup\{0\},$ we construct a special set $\A_n\in\F_2$ such that 
$\sup_{u\in \A_n}J(u)=\ga_2.$ Let $\{t_n\}$ be a sequence in $[0,1-r)$ such that $t_0=0$ and $ t_n\ra 1-r.$ 
For each $n\in \N\cup\{0\}$, let $B_n=B_r(t_n {e}_1)$ and $\Om_n=B_1\setminus \overline{B_n}$ where $e_1$
is the unit vector in the direction of the first coordinate axis.  Let $u_n,v_n$ be the respective 
first eigenfunctions on $B_n$  and  $\Om_n$ satisfying $u_n>0$ on $B_n$, $v_n>0$ on $\Om_n$ and $\norm{u_n}_p=\norm{v_n}_p=1$. By translation 
invariance of the $p$-Laplacian, we have $\la_1(B_n)=\ga_2.$ Further,          
from Theorem 1 of \cite{Anisa}, we also have $\la_1(\Om_n)\le \ga_2.$            
Let $\wide{u}_n$ and $\wide{v}_n$ be the zero extensions to the entire $B_1.$  For each $n\in \N\cup\{0\}$, we consider 
$$ \A_n:= \{ a \wide{u}_n+ b \wide{v}_n: |a|^p+|b|^p=1\}.$$ 
One can easily verify that $\A_n \in \F_2$ and  $ \sup_{u\in \A_n}J(u)=\ga_2,\forall n\in \N\cup\{0\}.$                                                                                                                                     
                                                                                   
Now we ask the question whether $\A_n$ contains a critical point of $J$ on $S$ or not. This leads to the following two alternatives:
 \begin{itemize}
  \item [(i)] for every $n\in \N,$ $\A_n$ contains at least one critical point of $J$ on $S,$
  \item [(ii)] there exists $n_0\in \N$ such that $\A_{n_0}$ does not contain any critical point of $J$ on $S.$
 \end{itemize}                                                                                                                                                                                                             
In the next lemma we show that alternative (i) does not hold.
\begin{lemma}
  Let $\A_n$ be as above. Then alternative {\rm (i)} does not hold.
\end{lemma}
  
\begin{proof}
Let $u_n$ and $\wide{u}_n$ be as above. Then $u_n(x)=u_0(x-t_n {e}_1)$ and hence the sequence $\{\wide{u}_n(x)\}$ converges to 
$u^*(x)=\wide{u}_0(x-(1-r){e}_1)$ both pointwise and in $\wp(B_1)$. On the other hand, the sequence $\{\wide{v}_n\}$ is bounded by $\ga_2$ in $\wp(B_1).$ 
Thus up to a subsequence, $\wide{v}_n$ converges to some $v^*$ weakly in $\wp(B_1)$ and a.e. in $B_1.$  
If alternative (i) holds, then we get a sequence $\{ \phi_n= a_n \wide{u}_n+ b_n \wide{v}_n: |a_n|^p+|b_n|^p=1 \}$ of eigenfunctions of \eqref{evp}  with 
eigenvalues  $J(\phi_n).$ By Proposition \ref{regularity}, the  eigenfunctions are in $\C^1(\overline{B_1})$ and hence we must have $a_n b_n <0.$  Now  we may 
assume that $a_n>0$ and $b_n<0$ for each $n.$  Further, the sequences $\{J(\phi_n)\}, \{a_n\}$ and $\{b_n\}$ are bounded. Thus for a subsequence we get 
$J(\phi_n)\ra \la^*$, $a_n \ra a^*$ and $b_n\ra b^*$ for some $\la^*,a^* \ge 0$ and $b^*\le 0.$  The sequence $\{\phi_n\}$ is bounded in $\wp(B_1)$ and hence  
up to  a subsequence $\phi_n\rightharpoonup \phi^*$ in $\wp(B_1)$ and a.e. in $B_1.$ Since $a_n \wide{u}_n+ b_n \wide{v}_n \ra a^* u^*+ b^* v^*$ a.e. in $B_1,$ 
we must have 
$$\phi^*=a^* u^*+ b^* v^*.$$  
Since,  each $\phi_n$ is an eigenfunction of \eqref{evp}, it is easy to verify that $\phi^*$ is an eigenfunction 
corresponding to the eigenvalue $\la^*.$ Thus by the regularity of $\phi^*,$ we 
must have $a^* b^*< 0$ and hence 
$$a^*>0,\quad b^*< 0.$$ 
Let $B^*=B_r((1-r){e}_1)$ and $\Om^*=B_1\setminus B^*.$ Clearly $u^*>0$ on $B^*$ and $u^*=0$ on 
$\Om^*$. On the other hand, $v^*=0$ a.e. in $B^*$ and $v^*\ge 0$ a.e. on $ \Om^*.$  Thus from the continuity of the $\phi^*$ we get 
$$\phi^*(x)> 0,\; \forall \,x \in B^*, \quad\quad  \phi^*(x)\le 0,\; \forall \, x\in\Om^*.$$  
Now we apply Theorem 5 of \cite{Vazquez} (a Hopf's lemma type result for $p$-Laplacian) on $B^*\cup \{e_1\}$ to get   
\begin{align*}
 \frac{\partial \phi^*}{\partial x_1}({e}_1)=c<0.
\end{align*}   
Since $\phi^* \le 0$ on $\Om^*$ we also have
\begin{align*}
 \frac{\partial \phi^*}{\partial \eta(x)}(x)\ge 0, \; \forall\, x\in \partial B_1 \setminus \{e_1\},
 \end{align*}
 where $\eta(x)$ is the outward unit normal to $B_1$ at $x.$
The above two inequalities contradicts the fact that $\phi^*$ is in $\C^1(\overline{B_1}).$ Thus we conclude that  alternative (i) does not hold.
\end{proof}

\noindent{\bf Proof of Theorem \ref{Theorem1}}
 Let $\A_n$ be as before. Thus we have $\sup_{v\in \A_n} J(v)\le \ga_2.$ By the above lemma, the alternative (ii) holds, i.e. there exists $n_0 \in \N$ such 
 that $\A_{n_0}$ does not contain any critical points of $J$ on $\S.$ Thus by Proposition \ref{deformation} and by \eqref{downflow}, we get $\wide{\A}\in \F_2$ 
 such that 
 $$\sup_{u\in \wide{\A}} J(u) < \sup_{v\in \A} J(v)\le \ga_2.$$
 Now from \eqref{var2} we get $\la_2<\ga_2.$
\section{Construction of nonradial eigenfunctions}
In this section we construct an infinite sequence of nonradial eigenfunctions of \eqref{evp}.
First we fix  the following conventions. A vector $x$ in $\R^N$ is always
taken as a $1
\times N$ row vector, i.e $x=(x_1,x_2\ldots x_N).$ The transpose of $x$, denoted by 
$x^T$, is an $N\times 1$ column vector. We denote the scalar product in $\R^N$ by $x\cdot y\, (=x y^T).$ 
Let $H$ be the hyperplane given by $H=\{x\in \R^N: x\cdot a=0\}$ for some
unit vector $a\in \R^N.$ Let
$\sigma_H$ be the reflection about $H.$ Then    
$$ \sigma_H(x)= x-2(x\cdot a)a = x  (I-2 a^T a).$$
Next we list some of the elementary properties of $\sigma_H$ that we use in
this article.
\begin{enumerate}[(i)]
\item $\sigma_H$ is linear and $\sigma_H=(I-2 a^T a)$.
\item ${\sigma_H}^{-1}=\sigma_H.$
\item $\sigma_H$ is  symmetric and orthogonal.
 \item $D \sigma_H(x)=\sigma_H$  and ${\rm det}D\sigma_H(x)= -1,\;\forall\,x \in
\R^N.$ 
 \end{enumerate}

Let $\O$ be a bounded domain symmetric about $H,$ i.e, $\sigma_H(\O)=\O.$ Let
$\O^+:=\{x\in \O: \inpr{x,a} >0\}$ and let $\O^-=\sigma_H(\O^+).$ Let $u\in
\wp(\O^+)$ be a weak solution of \eqref{genevp} on $\Om=\O^+.$ Define $u^*$ on
$\O$ as below
\begin{align*}
 u^*(x)&=\left\{\begin{array}{ll}
             u(x), \quad x\in \O^+, \\       
                         0,\quad  x\in \partial(\O^+)\cup\partial(\O^-),\\
-u(\sigma_H(x)),\quad x \in \O^- . \end{array}
 \right.
\end{align*}
Clearly $u^*\in \wp(\O)$ and we also have the following lemma:
\begin{lemma}
\label{gluing}
 Let $u^*$ be defined as above. Then $u^*$ is a weak solution
of \eqref{genevp} on $\Omega=\O.$
\end{lemma}

\begin{proof}
Let $\phi\in\wp(\O)$ be a test function.
We show that
\begin{equation}\label{eqn1}
\int_{\O} |\nabla u^*(x)|^{p-2} \nabla u^*(x) \cdot  \nabla \phi(x) \dx = \la
\int_{\O} |u^*(x)|^{p-2}u^*(x) \phi(x)\dx.
\end{equation}
From the definition of $u^*,$
\begin{align*}
\int_{\O} |\nabla u^*(x)|^{p-2} \nabla u^*(x) \cdot  \nabla \phi(x)
\dx&= \int_{\O^+} |\nabla
u(x)|^{p-2} \nabla u(x)\cdot\nabla \phi(x) \dx \\ &+   \int_{\O^-}
|\nabla(-u(\sigma_H(x)))|^{p-2}
\nabla (-u(\sigma_H(x)))\cdot \nabla \phi(x) \dx \\
\end{align*}
Now by noting that $D\sigma_H(x)=\sigma_H$ and $\sigma_H$ is an isometry  we
get 
\begin{align*}
\int_{\O^-}|\nabla(-u(\sigma_H(x)))|^{p-2}
&\nabla (-u(\sigma_H(x)))\cdot \!\nabla \phi(x)\dx \\
&=-\int_{\O^-}|\nabla u(\sigma_H(x)) \sigma_H |^{p-2}   [\Gr
u(\sigma_H(x)) \sigma_H ]\cdot \Gr \phi(x)\dx, \\
&=-\!\int_{\O^-} | \nabla u(\sigma_H(x)) |^{p-2}
\Gr u(\sigma_H(x))\cdot \! [\Gr \phi(x) \sigma_H] \dx,
\end{align*}
where the equality in the last step also uses the fact that $\sigma_H$ is
symmetric. Now the change of variable $y=\sigma_H(x)$ along with
properties (ii) and (iv) of $\sigma_H$ 
will give
\begin{align*}
\int_{\O^-}\!\!\!\!\!\! |\nabla(-u(\sigma_H(x)))|^{p-2}
\nabla (-u(\sigma_H(x))\cdot \!\nabla \phi(x)\dx
&=-\int_{\O^+}\!\!\! \!\! | \nabla u(y) |^{p-2}  \Gr
u(y)\cdot [\Gr \phi(\sigma_H(y)) \sigma_H] \dy.
\end{align*}
Thus
\begin{equation*}
\int_{\O} |\nabla u^*(x)|^{p-2} \nabla u^*(x) \cdot  \nabla \phi(x) \dx = 
\int_{\O^+} | \nabla u(x) |^{p-2} 
\Gr
u(x) \cdot [\nabla \phi(x) -[\Gr \phi(\sigma_H(x)) \sigma_H]]  dx.
\end{equation*}
Let $\psi(x)=\phi(x)- \phi(\sigma_H(x))$. Then we have
\begin{equation}\label{eqn2}
\int_{\O} |\nabla u^*(x)|^{p-2} \nabla u^*(x) \cdot  \nabla \phi(x) \dx 
=\int_{\O^+} | \nabla u(x) |^{p-2} 
\Gr
u(x) \cdot \nabla \psi(x)  dx.
\end{equation}
Further,
\begin{equation}\label{eqn3}
\int_{\O} |u^*(x)|^{p-2} u^*(x) \phi(x)
=\int_{\O^+} | u(x)|^{p-2} u(x)  \psi(x)\dx.
\end{equation}
Clearly $\psi\in
\wp(\O^+)$ and hence  
\begin{equation}\label{eqn4}
 \int_{\O} |\nabla u^*(x)|^{p-2} \nabla u^*(x) \cdot  \nabla \phi(x) \dx =
\int_{\O^+} | u(x)|^{p-2} u(x)  \psi(x)\dx,
\end{equation}
since $u$ solves \eqref{genevp} on $\Om=\O^+.$
Now \eqref{eqn1} follows from \eqref{eqn2},\eqref{eqn3} and \eqref{eqn4}.
\end{proof}

\noindent{\bf Proof of Theorem \ref{Theorem2}}:
For $n\in \N,$ we consider the sectors $S_k$ 
given by $S_k=\{x\in B_1: \frac{(k-1)\pi}{n}<\theta_{N-1}<\frac{k\pi}{n}
\},k=1,\ldots,n.$ Let
$H_k$ be the hyperplane given by $H_k=\{x\in \R^N:\theta_{N-1}=\frac{\pi
k}{n}\},$ for $k=1,...n$.
Let $\tau_n$ be the first eigenvalue for the  $p$-Laplacian on $S_1$
and $u_1(x)$ be a corresponding eigenfunction. 
For $i=2,\ldots,n,$ we define  $u_i$ recursively by
$u_i=-u_{i-1}(\sigma_{H_{i-1}}(x)),$ the odd reflection of
$u_i$ about $H_{i-1}$. 
Let $D^+$ be the sector given by $\{x\in B_1: 0<\theta_{N-1}<{\pi}
\}.$ Now we define $u^*$ on $\overline{D^+}$ by
$$ u^*(x)=u_i(x), \quad x\in \overline{S_i},\quad i=1,\ldots,n . $$
From Lemma \ref{gluing}, it is clear that $u^*$ solves \eqref{genevp} 
on the union of two adjacent sectors with $\la=\tau_n$. 
Let $U_i= \{x\in B_1: \frac{ (i-1)\pi}{n}<\theta_{N-1}<\frac{ (i+1) \pi }{n}
\},$ for $i=1,\ldots,n-1.$ Then  $\{U_i\}_{i=1}^{n-1}$ is an open covering
of $D^+.$ 
Let $\{\phi_i\}_{i=1}^{n-1}$ be a $\C^\infty$ partition of the unity
corresponding to this open covering. Note that for  each $i,$ $\phi_i$
intersects at most $S_i$ and $S_{i+1}.$
Since $\sum_{i=1}^{n-1}\phi_i=1,$ we have
\begin{align*}
\int_{D^+} |\nabla u^*(x)|^{p-2} \nabla u^*(x) \cdot  \nabla \phi(x)
\dx&= \int_{D^+} |\nabla u^*(x)|^{p-2} \nabla u^*(x) \cdot  \nabla \(
\phi(x)\sum_{i=1}^{n-1}\phi_i(x)\)\dx\\
&= \sum_{i=1}^{n-1}  \int_{D^+} |\nabla u^*(x)|^{p-2} \nabla u^*(x) \cdot 
\nabla (\phi(x) \phi_i(x))\dx.\\
\end{align*}
For a fixed $i$, the product $\phi \phi_i\in\wp(U_i).$ Hence by the
definition of $u^*$ and Lemma \ref{gluing}, we get 
\begin{align*}
 \int_{D^+} |\nabla u^*(x)|^{p-2} \nabla u^*(x) \cdot  \nabla (\phi(x)
\phi_i(x))\dx 
 &=\int_{U_i} |\nabla u^*(x)|^{p-2} \nabla u^*(x) \cdot  \nabla
(\phi(x) \phi_i(x))\dx,\\
&=\tau_n \int_{U_i} | u^*(x)|^{p-2} u^*(x)  (\phi(x)
\phi_i(x))\dx,\\
  &=\tau_n \int_{D^+} | u^*(x)|^{p-2} u^*(x)  (\phi(x) \phi_i(x))\dx.
 \end{align*}
 Thus we get 
 \begin{align*}
 \int_{D^+} |\nabla u^*(x)|^{p-2} \nabla u^*(x) \cdot  \nabla \phi(x) \dx
 &=\sum_{i=1}^{n-1}\tau_n \int_{D^+} | u^*(x)|^{p-2} u^*(x)  (\phi(x)
\phi_i(x))\dx, \\
 &=\tau_n\int_{D^+} |u^*(x)|^{p-2} u^*(x)  \(\sum_{i=1}^{n-1} \phi(x)
\phi_i(x)\)\dx, \\
 &=\tau_n\int_{D^+} |u^*(x)|^{p-2} u^*(x)  \phi(x) \dx.
  \end{align*}
  
 Now define $\Psi_n$ on $B_1$ by
\begin{align*}
 \Psi_n(x)&=\left\{\begin{array}{ll}
             u^*(x), \quad x\in D^+, \\       
                         0,\quad  x\in \partial(D^+)\cup\partial(D^-),\\
-u^*(\sigma_{H_0}(x)),\quad x \in D^- , \end{array}
 \right.
\end{align*}
where $D^-=\{x\in B_1: \pi<\theta_{N-1}<{2\pi}
\}$ is the ``lower'' half-ball and $H_0$ is the hyperplane
corresponding to $\theta_{N-1}=0$.
Applying Lemma \ref{gluing} once again, we get that $\Psi_n$ is a weak  solution
of \eqref{evp}.
Thus we have constructed an eigenpair $\{\tau_n,\Psi_n\}$ of \eqref{evp} such
that $\Psi_n$ has $2n$ nodal domains
and  each nodal domain is a sector with measure $\frac{|B_1|}{2n}$.
\qed \\

In the next remark we list some of the interesting open problems related
to the results of this paper:
\begin{remark}\rm{(Open problems associated with \eqref{evp})
\begin{enumerate}
 \item Payne conjectured (Conjecture 5, \cite{Payne-1}) that the nodal line
of a second eigenfunction of
Laplacian on a bounded domain $\Om \subset \R^2$ cannot be a closed curve.
In \cite{Payne-2}, he proved his conjecture for the special case when $\Om$ is
convex in $x$ and symmetric about $y$ axis.  For a ball, his result was easily
obtained by  applying the Fourier method to the Laplacian in polar
co-ordinates. 
We conjecture that the nodal surface of a second eigenfunction of
$p$-Laplacian on  $B_1$ cannot be a closed surface in $B_1$ for $1<p<\infty$
and for every $N\ge 2.$                
 \item For $p=2,$ it is easy to see that $\la_2=\tau_1$. We
anticipate the same result for $p\ne 2$ as well. More precisely, the
nodal surface of any second eigenfunction is given by the intersection of a
hyperspace with $B_1$ and the nodal domains are the half balls symmetric to this
hyperspace.
 \item We have just shown that all the eigenfunctions corresponding to $\la_2$ are
nonradial. Is it true that all the eigenfunctions corresponding to the second radial eigenvalue $\ga_2$
are radial? 

 \item Note that $\la_2$ is the least eigenvalue having an eigenfunction with
two nodal domains. For $p=2,$ it can also be  seen that $\ga_2$ is the maximal
eigenvalue having an eigenfunction with two nodal domains. In other words, the
eigenfunctions corresponding to
$\la>\ga_2$ must have at least three nodal domains. Is this true
for $p\ne2?$ 
\end{enumerate} }
\end{remark}
\bibliography{ref1}
\bibliographystyle{abbrv}
\end{document}